\documentclass[11pt]{article}
\usepackage{geometry}
\usepackage{amsmath}
\usepackage{amsfonts}
\usepackage{amssymb}
\usepackage{amsthm}
\usepackage{verbatim}
\usepackage{graphicx}
\usepackage{caption}
\usepackage{subcaption}

\usepackage{tikz}
\usetikzlibrary{arrows,calc}
\tikzset{
>=stealth',
help lines/.style={dashed, thick},
axis/.style={<->},
important line/.style={thick},
connection/.style={thick, dotted},
}

\newtheorem{defn}{Definition}
\newtheorem{lem}{Lemma}
\newtheorem{cor}{Corollary}
\newtheorem{remark}{Remark}
\newtheorem{thm}{Theorem}
\newtheorem{prop}{Proposition}
\newtheorem{assumption}{Assumption}

\title{Non-linear PDE Approach to Time-Inconsistent Optimal Stopping}
\author{Christopher W. Miller\footnotemark[1]}
\date{}

\begin{document}

\maketitle

\footnotetext[1]{Department of Mathematics, University of California, Berkeley. Supported in part by NSF GRFP under grant number DGE 1106400.}

\begin{abstract}
We present a novel method for solving a class of time-inconsistent optimal stopping problems by reducing them to a family of standard stochastic optimal control problems. In particular, we convert an optimal stopping problem with a non-linear function of the expected stopping time in the objective into optimization over an auxiliary value function for a standard stochastic control problem with an additional state variable. This approach differs from the previous literature which primarily employs Lagrange multiplier methods or relies on exact solutions. In contrast, we characterize the auxiliary value function as the unique viscosity solution of a non-linear elliptic PDE which satisfies certain growth constraints and investigate basic regularity properties. We demonstrate a connection between optimal stopping times for the original problem and optimal controls of the auxiliary control problem. Finally, we discuss extensions to more general dynamics and time-inconsistencies, as well as potential connections to degenerate Monge-Ampere equations.
\end{abstract}



\pagestyle{myheadings}
\thispagestyle{plain}
\markboth{C. W. Miller}{PDE Approach to Time-Inconsistent Optimal Stopping}

\section{Introduction}\label{Section:Introduction}

Time-inconsistent optimal stopping and optimal stochastic control problems are characterized by the failure of standard dynamic programming arguments to apply. Time-inconsistency usually results in practice from a non-linearity in the objective function, and has previously been studied in the economic and financial context primarily related to non-exponential discounting and mean-variance portfolio selection (see \cite{Bjork2014,PedersenPeskir2013,Yong2012,Rudloff2013}). The two common approaches to dealing with time-inconsistency in the literature are to reformulate the problem as a time-consistent problem, while possibly changing the value function, or to employ a ``pre-commitment strategy'' which may need to be recomputed for each new initial condition.

In this paper, we present a novel method for solving a class of time-inconsistent optimal stopping problems by reducing them to a family of standard stochastic control problems. We emphasize the following three-step strategy for dealing with time-inconsistency: (1) Condition on the time-inconsistent feature to obtain a constrained problem, (2) embed the constrained problem in a time-consistent problem in a higher-dimensional state-space, and finally, (3) construct an optimal control of the original problem by starting at an optimal choice of the new state variables. This procedure allows us to compute the value function of a pre-commitment strategy, and construct optimal stopping times under suitable regularity. While the first and last step are well-understood in the previous literature (see \cite{PedersenPeskir2013,PedersenPeskir2016}), the main contribution of this paper is our second step, which is typically replaced by an application of Lagrange multipliers or an appeal to exact solutions.

The particular type of time-inconsistency featured in this paper is a non-linear function of the expected stopping time appearing in the objective function. Following the ideas developed in \cite{PedersenPeskir2013}, we condition on the expected value of the stopping time to obtain an expectation-constrained optimal stopping problem. In contrast with the previous literature, we embed the constrained optimal stopping problem into a time-consistent control problem with one extra state variable rather than employing the method of Lagrange multipliers. We characterize this auxiliary value function as the viscosity solution of a degenerate-elliptic Hamilton-Jacobi-Bellman (HJB) equation subject to certain growth constraints. Furthermore, we demonstrate how to construct an optimal stopping time, subject to regularity conditions, for the original problem by solving a sequential concave optimization problem.

The overarching idea of our approach is inspired by that reported in \cite{PedersenPeskir2013}. In that paper, the authors solve a mean-variance stopping problem with a similar non-linearity by conditioning on the time-inconsistent feature and solving the resulting constrained stopping problem with free-boundary techniques and verification arguments along the lines of \cite{PeskirShiryaev2006}. In comparison, the reader may view the main contribution of this paper as a novel solution of the expectation-constrained optimal stopping problem by embedding in a time-consistent stochastic control problem.

The investigation of constrained optimal stopping problems is not new, but most previous work focuses on Lagrange multiplier approaches to the constraint (see \cite{Kennedy1982,LopezSanMiguel1995,Horiguchi2001,Makasu2009}). In contrast, we identify the expectation-constrained auxiliary value function as the unique viscosity solution of degenerate-elliptic Hamilton-Jacobi-Bellman equation, subject to certain growth constraints. The main advantage of our approach is that it depends neither on the specific form of non-linearity in the problem, nor on the availability of analytic solutions. When analytic solutions are not available, a solution via the method of Lagrange multipliers generally relies on a numerical optimization of the Lagrange dual problem. While effective, this approach can be unstable in practice since we have no regularity estimates on the dual function apart from convexity, and computation of the sub-gradient is often subject to truncation error. In contrast, the non-linear elliptic PDE featured in this paper has established numerical approximation schemes with guaranteed convergence and, more importantly, stability (see \cite{Oberman2008}). 

The ideas developed in this paper can be extended to deal with other types of time-inconsistent features in optimal stopping and optimal stochastic control problems. We primarily emphasize the case of optimal stopping due to its relative technical and pedagogical simplicity. We briefly discuss extensions to diffusion processes and more general types of time-inconsistencies in the final section, albeit formally. The ideas are related to recent solutions of mean-variance portfolio optimization, optimal control under certain non-linear risk-measures, and distribution-constrained optimal stopping (see \cite{PedersenPeskir2016,MillerYang2015,BayraktarMiller2016}).

Our approach is similar to the dynamic approach of \cite{KarnamMaZhang2016}, wherein the authors introduce extra state variables to remove time-inconsistency introduced by a system of controlled backwards stochastic differential equations. Similar to this paper, the authors convert a problem without an immediate dynamic nature to a dynamic problem with additional state variables. There appear to be additional analogies with the formal generalizations provided in Section~\ref{Section:Generalization}. In terms of a focus on time-inconsistent optimal stopping problems, this paper is similar to \cite{XuZhou2013}, where the authors consider a non-linear functional of a stopped process as the objective function. However, the method of solution differs entirely. Whereas these authors relate their problem to optimization over the distribution of the stopped process using Skorokhod embedding, we convert to a dynamic problem in a larger state-space.

We note that since the original circulation of this paper, the same elliptic PDE has been obtained independently in \cite{AnkirchnerKleinKruse2015} in a direct analysis of the expectation-constrained optimal stopping problem. The authors place relatively more emphasis on further analysis of the expectation-constrained optimal stopping problem, which may complement a reading of this paper by providing, for instance, specific examples of various degenerate behavior in the problem. In contrast, this paper spends more time on the analysis of the outer-optimization problem which is unique to our class of time-inconsistent optimal stopping problems. In our paper, we largely relegate consideration of degeneracies or regularity issues from the forefront by the use of viscosity solutions in the analysis of the expectation-constrained optimal stopping problem.

Some other notable works in the literature on time-inconsistent problems include \cite{Bjork2014,Hu2012,Yong2012,Ekeland2010}. Most of the previous literature focuses either on specific examples of time-inconsistency (often arising from non-exponential discounting or mean-variance optimization) or on notions of equilibrium strategies, which view time-inconsistent problems as a sequential game against ones future self. In an equilibrium strategy, the optimal equilibrium strategy can be characterized as the solution to an ``extended HJB'' system. In general, these systems feature multiple solutions and exhibit values strictly less than the value function of a pre-committed strategy. The price we must pay in our approach is that the entire value function must be recomputed if we change the initial conditions. This point is related to the notions of static and dynamic optimality which are explored in depth by \cite{PedersenPeskir2013} and \cite{PedersenPeskir2016}.

This paper proceeds as follows: In Section \ref{Section:Formulation}, we formulate a time-inconsistent optimal stopping problem featuring a non-linearity in the expectation of the stopping time. In Section \ref{Section:EquivalentProblem}, we show how this is equivalent to a sequential optimization problem which is time-consistent at each step. We characterize the value function of the time-consistent subproblem as the viscosity solution of a HJB equation. In Section \ref{Section:Properties}, we prove various bounds and concavity properties of the auxiliary value function and characterize it as the unique viscosity solution of an HJB subject to certain growth constraints. We also show that the sequential optimization problem is concave, which has practical implications in terms of regularity of the value function and existence of maximizers. Finally, in Section \ref{Section:Generalization}, we provide remarks on connections and generalizations of this approach. In particular, we show how to generalize the underlying dynamics and handle more general non-linearities. We also highlight a connection between our HJB and the Monge-Ampere equation.

\section{Problem Formulation}\label{Section:Formulation}

Let $W$ be a standard Brownian motion on a probability space $(\Omega,\mathcal{F},\mathbb{P})$ with filtration $\mathbb{F}=\{\mathcal{F}_t\}_{t\geq 0}$ satisfying the usual properties. Let $\mathbb{P}^x$ denote the measure corresponding to $W$ starting from a given point $x$. Denote by $\mathcal{T}$ the collection of all $\mathbb{F}$-stopping times such that $\mathbb{E}^x\left[\tau^2\right]<\infty$.\footnote{It is worth noting that the restriction to finite-variance stopping times is not a natural condition in traditional optimal stopping applications and may lead to lack of an optimal stopping time. However, it is the natural restriction within our analysis, which relies heavily on the application of the martingale representation theorem. In Section \ref{Subsection:Construction}, we discuss the construction of $\epsilon$-suboptimal stopping times in cases that an optimal stopping time fails to exist.}

Let $f:\mathbb{R}\to\mathbb{R}$ be continuous and bounded from above\footnote{Note we require the Brownian motion to be one-dimensional here, although the key results carry over to multi-dimensional versions easily. For a formal conversation on higher-dimensional analogues, refer to Section~\ref{Subsection:Extension}.}. Let $g:\mathbb{R}^+\to\mathbb{R}$ be continuous. Under these conditions, we define a criterion as
\begin{equation}\nonumber
J\left[x,\tau\right] := \mathbb{E}^x\left[ f(W(\tau))\right] + g\left(\mathbb{E}^x\left[\tau\right]\right).
\end{equation}

\begin{defn}
The time-inconsistent optimal stopping problem is defined by
\begin{equation}\label{Equation:MainProblem}
v(x) := \sup\limits_{\tau\in\mathcal{T}} J\left[x,\tau\right]\hspace{0.5cm}\text{for all }x\in\mathbb{R}.\end{equation}\end{defn}

The main result of this paper is that we can re-cast this time-inconsistent stopping problem as a sequential optimization problem\footnote{Here, we refer to the problem as ``sequential'' in the sense that there are two optimization problems which are solved sequentially. First, we solve an auxiliary stochastic optimal control problem. Then, we optimize over starting values of the additional state variable.} involving a time-consistent stopping problem. In particular, let $\mathcal{A}$ represent the collection of all square-integrable, progressively-measurable processes and consider a time-consistent stochastic control problem.
\begin{defn}We define the auxiliary time-consistent control problem as
\begin{equation}\label{Defn:AuxProblem}w(x,y):=\left\{\begin{array}{rl}
\sup\limits_{\alpha\in\mathcal{A}}&\mathbb{E}^{x,y}\left[f(W(\tau^\alpha))\right]\\
& dY^\alpha = -dt + \alpha dW,\,Y^\alpha_0=y\\
& \tau^\alpha = \inf\left\{t\geq 0\mid Y^\alpha(t)=0\right\}
\end{array}\right.\end{equation}
for all $(x,y)\in\mathbb{R}\times [0,\infty)$.\end{defn}

The focus of this paper is on the following equivalence between $v$ and optimization over $w$.
\begin{equation}\label{Eqn:SeqOptimization}
v(x) = \sup\limits_{y\geq 0}\left[w(x,y) + g(y)\right].
\end{equation}
Under certain growth assumptions on the functions $f$ and $g$, we show the following results:
\begin{enumerate}
\item The auxiliary value function $w$ is the unique viscosity solution of an HJB with certain growth restrictions (Theorem \ref{Thm:UniqueHJB}),
\item The sequential optimization problem (\ref{Eqn:SeqOptimization}) is concave in $y$ (Theorem \ref{Thm:MaximalElement}), and
\item Under additional regularity conditions, we can construct an optimal stopping time $\tau$ for the original problem from an optimal control in the auxiliary control problem (Theorem \ref{Thm:StoppingTimeConstruction}).
\end{enumerate}
We also note that the three-step procedure we follow in this problem suggests a more general approach to other kinds of non-linearities in time-inconsistent optimal stopping and optimal stochastic control problems. We formally highlight this procedure in more generality in Section \ref{Subsection:GeneralInconsistentcies}.

\section{Equivalent Sequential Time-Consistent Problem}\label{Section:EquivalentProblem}

Our goal is to convert this time-inconsistent optimal stopping problem into a sequential optimization problem involving a time-consistent control problem. Our general approach to time-inconsistency is the following:
\begin{itemize}
\item Step 1: Condition on any time-inconsistent features in the problem to generate a family of optimal-stopping problems with constraints,
\item Step 2: Enlarge the state-space to embed the constrained problems in a time-consistent problem, and
\item Step 3: Construct an optimal stopping time for the time-inconsistent problem by picking an optimal value of the time-inconsistent feature and generating an optimal solution to the time-consistent problem starting from that choice.
\end{itemize} 
In our particular problem, this will involve adding a new state variable to track the expectation of the optimal stopping time. As the system evolves, we expect this variable to be a super-martingale as the expected time until stopping drifts downward. However, it's possible to allow this expected stopping time to increase along certain paths so-long as it is compensated by a decrease along other paths.

\subsection{Conditioning on Time-Inconsistent Features}

For any $(x,y)\in\mathbb{R}\times[0,\infty)$, consider the following subset of stopping times:
\begin{equation}\nonumber
\mathcal{T}_{x,y} := \left\{\tau\in\mathcal{T}\mid\mathbb{E}^x\left[\tau\right]=y\right\}.
\end{equation}

Furthermore, consider the following family of constrained optimal stopping problems.

\begin{defn}\label{Defn:ConstrainedFunctions}Define a family of constrained optimal stopping problems by
\[\tilde{w}(x,y) := \sup\limits_{\tau\in\mathcal{T}_{x,y}} \mathbb{E}^x\left[f(W(\tau))\right]\hspace{0.5cm}\text{for all }(x,y)\in\mathbb{R}\times [0,\infty).\]\end{defn}

We first claim that we can reformulate the time-inconsistent optimal stopping problem \eqref{Equation:MainProblem} as a sequential optimization problem involving these constrained optimal stopping problems.

\begin{thm}[\cite{PedersenPeskir2013}]
\label{Thm:AuxEquivalence}We can reformulate $v$ in terms of $\tilde{w}$ as
\[v(x) = \sup\limits_{y\geq 0} \left[\tilde{w}(x,y) + g(y)\right]\hspace{0.5cm}\text{for all }(x,y)\in\mathbb{R}\times [0,\infty).\]\end{thm}
\begin{proof}
The key is to note that for each $x\in\mathbb{R}$,
\[\mathcal{T} = \bigcup_{y\geq 0}\mathcal{T}_{x,y}.\]
Then we check that
\begin{eqnarray}
v(x) & = & \sup\limits_{\tau\in\mathcal{T}} J\left[x,\tau\right]\nonumber\\
& = & \sup\limits_{y\geq 0}\sup\limits_{\tau\in\mathcal{T}_{x,y}} \left[\mathbb{E}^x\left[f(W(\tau ))\right]+ g\left(\mathbb{E}^x\left[\tau\right]\right)\right]\nonumber\\
& = & \sup\limits_{y\geq 0}\sup\limits_{\tau\in\mathcal{T}_{x,y}}\left[\mathbb{E}\left[f(W(\tau ))\right]+ g(y)\right]\nonumber\\
& = & \sup\limits_{y\geq 0}\left[\tilde{w}(x,y)+g(y)\right].\nonumber
\end{eqnarray}
\end{proof}

\subsection{Equivalence to Time-Consistent Optimal Control}

Next, we reformulate the constrained optimal stopping problem as a time-consistent stochastic control problem. Let $\mathcal{A}$ be the set of all real-valued, progressively-measurable, and square-integrable processes. We begin this section by identifying $\tau\in\mathcal{T}_{x,y}$ with a control in $\mathcal{A}$.

\begin{lem}\label{Lem:MRT}
Let $\tau\in\mathcal{T}_{x,y}$. Then there exists $\alpha\in\mathcal{A}$ such that
\[\tau = \inf\left\{t\geq 0\mid y-t+\int_0^t\alpha(s)dW(s)=0\right\}\hspace{0.5cm}\mathbb{P}^x\text{-a.s.}\]
Similarly, let $\alpha\in\mathcal{A}$ and define
\[\tau^\alpha := \inf\left\{t\geq 0\mid y-t+\int_0^t \alpha(s)dW(s)=0\right\}.\]
Then $\tau^\alpha \in \mathcal{T}_{x,y}$.
\end{lem}
\begin{proof}
\begin{enumerate}
\item Because $\tau\in\mathcal{T}_{x,y}$ is a square-integrable random variable with expectation $y$, there exists a progressively-measurable, square-integrable process $\alpha$ such that
\[\tau = y + \int_0^\infty\alpha(s)dW(s)\hspace{0.5cm}\mathbb{P}^x\text{-a.s.}\]
by the martingale representation theorem. Taking conditional expectation with respect to $\mathcal{F}_\tau$, we see
\[\tau = y + \int_0^\tau\alpha(s)dW(s)\hspace{0.5cm}\mathbb{P}^x\text{-a.s.}\]

Define a stopping time
\[\sigma:=\inf\left\{t\geq 0\mid y-t+\int_0^t\alpha(s)dW(s)=0\right\}.\]
Then it is clear that $\sigma\leq\tau$ ($\mathbb{P}^x$-a.s.). Suppose that $\mathbb{P}^x\left[\sigma <\tau\right]>0$. Taking expectations of the martingale representation result conditional on $\mathcal{F}_\sigma$, we see
\[\mathbb{E}^x\left[\tau\mid\mathcal{F}_\sigma\right] = y + \int_0^\sigma\alpha(s)dW(s).\]
But by the definition of $\sigma$, this implies
\[\mathbb{E}^x\left[\tau-\sigma\mid\mathcal{F}_\sigma\right] = y - \sigma + \int_0^\sigma\alpha(s)dW(s) = 0.\]
Taking expectation, we see $\mathbb{E}^x\left[\tau\right]=\mathbb{E}^x\left[\sigma\right],$ which contradicts the assumption that $\mathbb{P}^x\left[\sigma<\tau\right]>0$. Then $\sigma=\tau$ $\mathbb{P}^x$-almost surely.

\item Fix $\alpha \in \mathcal{A}$ and define $\tau^\alpha := \inf\left\{t\geq 0\mid y-t+\int_0^t \alpha(s)dW(s)\right\}$. Then
\[y - \tau^\alpha + \int_0^{\tau^\alpha}\alpha(s)dW(s) = 0,\,\mathbb{P}^x\text{-a.s.}\]
Because $\alpha$ is assumed square-integrable, it follows that $\tau^\alpha$ is finite almost-surely. Then we can compute $\mathbb{E}^x\left[\tau^\alpha\right]=y$ and, furthermore,
\[\mathbb{E}^x\left[\left(\tau^\alpha\right)^2\right] = y^2 + \mathbb{E}^x\left[\int_0^{\tau^\alpha}\alpha(s)^2ds\right] \leq y^2 + \mathbb{E}^x\left[\int_0^\infty \alpha(s)^2ds\right] < +\infty.\]
Then $\tau^\alpha \in \mathcal{T}_{x,y}$.

\end{enumerate}
\end{proof}

The key to the main result is to convert between stopping times in $\mathcal{T}_{x,y}$ and controls in $\mathcal{A}$ via Lemma \ref{Lem:MRT} and instead view $w$ as the value function of a stochastic optimal control problem.

\begin{thm}\label{Thm:Equivalence}
Recall the stochastic control problem (\ref{Defn:AuxProblem}). We have the equivalence
\[w(x,y) = \tilde{w}(x,y).\]\end{thm}
\begin{proof}
Fix any $(x,y)\in\mathbb{R}\times [0,\infty)$.
\begin{enumerate}
\item Let $\alpha\in\mathcal{A}$ be an arbitrary control and consider the stopping time
\[\tau^\alpha := \inf\left\{t\geq 0\mid Y^\alpha(t)=0\right\}.\]
By Lemma~\ref{Lem:MRT} we have $\tau^\alpha \in \mathcal{T}_{x,y}$. Then
\[\tilde{w}(x,y)\leq\sup\limits_{\alpha\in\mathcal{A}}\mathbb{E}^{x,y}\left[f(W(\tau^\alpha))\right] = w(x,y),\]
because $\alpha$ was arbitrary.

\item Let $\tau\in\mathcal{T}_{x,y}$ be an arbitrary stopping time. By Lemma \ref{Lem:MRT} there exists a control $\alpha\in\mathcal{A}$ such that
\[\tau = \inf\left\{t\geq 0\mid y-t+\int_0^t\alpha(s)dW(s)=0\right\}\hspace*{0.5cm}\mathbb{P}^x\text{-a.s.}\]
But $\tau=\tau^\alpha$ almost surely if we let $Y^\alpha$ evolve via the control $\alpha$. Then
\[\tilde{w}(x,y)=\sup\limits_{\tau\in\mathcal{T}_{x,y}}\mathbb{E}^x\left[f(W(\tau))\right]\geq w(x,y),\]
because $\tau$ was arbitrary.
\end{enumerate}
\end{proof}

Now we can immediately apply dynamic programming principle to $w$ and show that it is a viscosity solution to an HJB equation.

\begin{cor}\label{Cor:ViscositySolution}
Assume $w$ is locally bounded. Then it is a viscosity solution of the HJB
\begin{equation}\label{Eqn:AuxiliaryHJB}
\left\{\begin{array}{rl}
u_y - \sup\limits_{\alpha\in\mathbb{R}}\left[\frac{1}{2}u_{xx}+\alpha u_{xy}+\frac{1}{2}\alpha^2 u_{yy}\right] = 0 & \text{in }\mathbb{R}\times(0,\infty)\\
u = f & \text{on }\mathbb{R}\times\{y=0\}.
\end{array}\right.\end{equation}\end{cor}

\begin{remark}
Special care must be taken when interpreting viscosity solutions to \eqref{Eqn:AuxiliaryHJB} because the unbounded diffusion coefficient may cause the Hamiltonian to be infinite. In this special case, the definition of a viscosity solution must be modified slightly. In this paper, we follow the framework provided in Definition 2.1 of \cite{DaLioLey2010}. Our later comparison results will also be proved from this perspective.
\end{remark}

\begin{remark}
An immediate question the reader may have is whether these infinite-horizon results extend to a finite-horizon problem. The key technical difficulty lies in determining a finite-horizon analogue of Lemma~\ref{Lem:MRT} and Theorem~\ref{Thm:Equivalence}.

If we constrain a stopping time $\tau\in\mathcal{T}_{x,y}$ to also satisfy $\tau\leq T$, almost-surely, then the argument in Lemma~\ref{Lem:MRT} still holds. However, given a control $\alpha\in\mathcal{A}$, the corresponding stopping time $\tau^\alpha$ will not satisfy $\tau^\alpha \leq T$, almost surely. Instead, the proper analogy is to restrict the collection of admissible controls $\mathcal{A}$ to only those which satisfy
\[Y^\alpha(t) \not\in (0,T-t) \implies \alpha(t) = 0.\]
The idea is that if $Y^\alpha(t) > T-t$, then the almost-sure and expectation constraints cannot be simultaneously satisfied. Therefore, we must exclude controls which allow this case. One can then see that an analogue of Lemma~\ref{Lem:MRT} holds for this subcollection of controls. The resulting analogue of HJB \eqref{Eqn:AuxiliaryHJB} will involve an extra state-variable and default to $\alpha\equiv 0$ in certain regions of state-space.
\end{remark}

\subsection{Construction of Optimal Stopping Times}\label{Subsection:Construction}

We have shown that we can recover the value function of the time-inconsistent optimal stopping problem by performing sequential optimization over choice of $y\geq 0$ and control $\alpha\in\mathcal{A}$. However, it remains to be shown that we can construct an optimal stopping time for the time-inconsistent problem.

The goal of this section is to relate the construction of optimal stopping times (or more generally, $\epsilon$-suboptimal stopping times) to the analogous problem of constructing optimal controls (or $\epsilon$-suboptimal controls). We start by recording a connection between $\epsilon$-optimal solutions of \eqref{Eqn:SeqOptimization} to $(2\epsilon)$-optimal solutions of \eqref{Equation:MainProblem}.

\begin{thm}\label{Thm:StoppingTimeConstruction}
Fix $x\in\mathbb{R}$ and $\epsilon > 0$. Let $y\geq 0$ be $\epsilon$-suboptimal for the optimization problem (\ref{Eqn:SeqOptimization}) and $\alpha\in\mathcal{A}$ be $\epsilon$-suboptimal for $w(x,y)$. Consider the stopping time
\[\tau^\alpha := \inf\left\{t\geq 0\mid y-t+\int_0^t\alpha(s)dW(s) = 0\right\}.\]
Then $\tau^\alpha\in\mathcal{T}$ is $(2\epsilon)$-suboptimal for the time-inconsistent problem \eqref{Equation:MainProblem}.\end{thm}
\begin{proof}
Fix $\epsilon >0$ and let $y\geq 0$ be $\epsilon$-suboptimal for problem (\ref{Eqn:SeqOptimization}). That is,
\[v(x)\leq w(x,y)+g(y)+\epsilon.\]
Furthermore, let $\alpha\in\mathcal{A}$ be $\epsilon$-suboptimal for $w(x,y)$. That is,
\[w(x,y)\leq\mathbb{E}^{x,y}\left[f(W(\tau^\alpha))\right]+\epsilon.\]
Recall from Lemma~\ref{Lem:MRT} that the stopping time $\tau^\alpha \in \mathcal{T}_{x,y}$ is admissible. Then, putting the two inequalities above together, we find
\begin{eqnarray}
v(x)& \leq & \mathbb{E}^x\left[f(W(\tau^\alpha))\right]+g(y)+2\epsilon \nonumber\\
& =&  \mathbb{E}^x\left[f(W(\tau^\alpha))\right]+g(\mathbb{E}^x\left[\tau^\alpha\right]) + 2\epsilon\nonumber\\
& = & J\left[x,\tau^\alpha\right] + 2\epsilon.\nonumber
\end{eqnarray}
Then $\tau^\alpha$ is $\left(2\epsilon\right)$-suboptimal for the time-inconsistent problem.
\end{proof}

There are several implications of this result.
\begin{itemize}
\item First, this result shows that if the sequential problem has a maximizing element, then it corresponds to an optimal stopping time for the time-inconsistent stopping problem. We demonstrate conditions for the outer-optimization problem to have a maximizing element in Section \ref{Section:Concavity}. Then the existence of an optimal control in the inner problem generally can be related to the existence of an optimal control for the auxiliary control problem. For instance, if the solution is well-behaved enough to construct an optimal Markov control, then it corresponds to an optimal stopping time through this result.

\item Second, this result demonstrates how, in the potential absence of an optimal stopping time for \eqref{Equation:MainProblem}, we can construct $(2\epsilon)$-suboptimal stopping times from $\epsilon$-suboptimal controls. Some care should be taken here because in general the resulting stopping times can, in principle, be heavily path-dependent at this point. Regardless, the problem of construction of $\epsilon$-suboptimal Markov controls is well-studied. We refer to the reader worried about practical construction of $\epsilon$-suboptimal Markov controls to \cite{Krylov2009}, for example.
\end{itemize}

\section{Viscosity Characterization and Concavity Properties}\label{Section:Properties}

In this section we prove various properties of the auxiliary value function $w$. In particular, we characterize it as the unique viscosity solution of (\ref{Eqn:AuxiliaryHJB}) with certain growth conditions. Furthermore, we highlight conditions under which there is concavity in the time-direction for each fixed value of $x\in\mathbb{R}$. Finally, we use this to show that the sequential optimization problem is a concave optimization problem.

In the following sections, we impose an asymptotic growth assumption.
\begin{assumption}\label{Assumption:Growth}
There exists $\beta,\gamma\geq 0$ such that $|f(x)+\beta x^2|\leq\gamma(1+|x|)$ for all $x\in\mathbb{R}$.
\end{assumption}
While this growth assumption may seem non-intuitive, it is chosen because we can demonstrate that the auxiliary value function $w$ decomposes into the sum of a fixed quadratic component and a function which is Lipschitz continuous in $(x,y)$. Because (\ref{Eqn:AuxiliaryHJB}) has an unbounded diffusion coefficient, it is not known to have a unique viscosity solution among functions with super-linear asymptotic growth (see \cite{DaLioLey2010}). However, relying on the decomposition above, we can demonstrate a form of uniqueness. If we allow more general growth (e.g. higher-order polynomials) then the residual component of $w$ is no longer Lipschitz in $y$ and our arguments no longer hold.

\subsection{Viscosity Solution Characterization}\label{Subsection:Viscosity}

Our goal in this section is to characterize the auxiliary value function $w$ as the unique viscosity solution of (\ref{Eqn:AuxiliaryHJB}) with a particular type of growth at infinity. We proceed by first obtaining growth bounds on the auxiliary value function.

\begin{lem}\label{Lem:AuxiliaryValueBounds}
Suppose that Assumption~\ref{Assumption:Growth} holds. Then there exists $C\geq 0$ such that
\[|w(x,y)-f(x)|\leq C(1+|x|+|y|)\hspace{0.5cm}\text{for all }(x,y)\in\mathbb{R}\times[0,\infty).\]\end{lem}
\begin{proof}
Fix $(x,y)\in\mathbb{R}\times[0,\infty)$ and let $\tau\in\mathcal{T}_y$ be an arbitrary stopping time. Then
\begin{eqnarray}
\mathbb{E}^x\left[f(W(\tau))\right] & \leq & \gamma+\gamma\mathbb{E}^x\left[|W(\tau)|\right]-\beta\mathbb{E}^x\left[W(\tau)^2\right]\nonumber\\
& \leq & \gamma+\gamma|x|-\beta x^2+\gamma\mathbb{E}^x\left[\sqrt{\tau}\right]-\beta\mathbb{E}^x\left[\tau\right]\nonumber\\
& \leq & \gamma(1+|x|)-\beta x^2+ (1-\beta)y+\frac{\gamma^2}{4}.\nonumber
\end{eqnarray}
Note, we used Tanaka's formula (see \cite{KaratzasShreve1991}) to bound the term $\mathbb{E}^x\left[|W(\tau)|\right]$. Because $\tau$ was arbitrary
\[w(x,y)\leq\gamma(1+|x|)-\beta x^2+(1-\beta)y+\frac{\gamma^2}{4}\leq f(x)+2\gamma(1+|x|)+(1-\beta) y+\frac{\gamma^2}{4}.\]
On the other hand, if we choose $\tau\equiv y$, then we compute
\begin{eqnarray}
w(x,y) & \geq & \mathbb{E}^x\left[f(W(y))\right]\nonumber\\
& \geq & -\gamma-\gamma\mathbb{E}^x\left[|W(y)|\right]-\beta\mathbb{E}^x\left[W(y)^2\right]\nonumber\\
& \geq & -\gamma-\gamma|x|-\beta x^2-\gamma\sqrt{y}-\beta y\nonumber\\
& \geq & f(x)-2\gamma(1+|x|)-(1+\beta) y-\frac{\gamma^2}{4}.\nonumber
\end{eqnarray}
Then we conclude
\[|w(x,y)-f(x)|\leq 2\gamma(1+|x|)+(1+\beta)y+\frac{\gamma^2}{4}\]
for all $(x,y)\in\mathbb{R}\times[0,\infty)$. Taking $C:=1+2\gamma+\beta+\gamma^2/4$, the result follows.
\end{proof}

The growth bounds above immediately show that the auxiliary value function $w$ is locally bounded. Then it is a standard procedure to check that $w$ is a viscosity solution of (\ref{Eqn:AuxiliaryHJB}), see \cite{Touzi2013}.

The question of uniqueness is a bit more subtle because this is a degenerate elliptic equation with no upper bound on the diffusion coefficient. Furthermore, we expect $w$ to have quadratic growth. Proving a standard comparison principle with quadratic growth and unbounded diffusion coefficient can be very technical (see \cite{DaLioLey2010}). Instead, we focus on the fact that we know that $w-f$ has linear growth.

\begin{lem}[Comparison Principle]\label{Lem:ComparisonPrinciple}
Let $u$ be an upper semi-continuous viscosity sub-solution and $v$ be a lower semi-continuous viscosity super-solution of (\ref{Eqn:AuxiliaryHJB}) respectively. Suppose that their difference has linear growth. That is, there exists $C\geq 0$ such that
\[|u(x,y)-v(x,y)|\leq C(1+|x|+|y|).\]
Then if $u\leq v$ on $\mathbb{R}\times\{y=0\}$, then $u\leq v$ on $\mathbb{R}\times [0,\infty)$.\end{lem}

This result is fairly standard, but is complicated by the unbounded diffusion coefficient. We provide a complete proof of Lemma \ref{Lem:ComparisonPrinciple} in the Appendix.

\begin{thm}\label{Thm:UniqueHJB}
Suppose that Assumption~\ref{Assumption:Growth} holds. Then the auxiliary value function $w$ is the unique viscosity solution of (\ref{Eqn:AuxiliaryHJB}) such that $w(x,y)-f(x)$ has linear growth.\end{thm}
\begin{proof}
We showed in Lemma \ref{Lem:AuxiliaryValueBounds} that $w$ is a viscosity solution satisfying the growth condition above. Suppose that $v$ is another viscosity solution such that $v-f$ has linear growth. Then there exists $C\geq 0$ such that
\[|w(x,y)-v(x,y)|\leq |w(x,y)-f(x)|+|v(x,y)-f(x)|\leq C(1+|x|+|y|).\]
Furthermore, $v=w$ on $\mathbb{R}\times\{y=0\}$, so it follows from Lemma \ref{Lem:ComparisonPrinciple} that $v=w$ on $\mathbb{R}\times[0,\infty)$.
\end{proof}

It is worth noting that if $f$ is sufficiently smooth, we can instead characterize $\tilde{w}:=w-f$ as the unique linear growth viscosity solution of a related HJB. This approach would avoid many of the technical difficulties above, but may require more regularity of $f$ than is available in practice.

\subsection{Concavity Properties}\label{Section:Concavity}

We propose conditions under which the auxiliary value function $w$ is concave in $y$ for each $x$. Note that because $w(x,0)=f(x)$, there is no hope of joint concavity of $w$ when $f$ is not concave. We also proceed to show conditions under which the sequential optimization problem for $v$ is concave.

\begin{prop}\label{Prop:ConcaveValueFunction}
Suppose that Assumption~\ref{Assumption:Growth} holds. Then for each $x\in\mathbb{R}$, the map $y\mapsto w(x,y)$ is concave.\end{prop}
\begin{proof}
By the results of Section \ref{Subsection:Viscosity}, we know $w$ is a viscosity solution of (\ref{Eqn:AuxiliaryHJB}). In particular, we claim it is a viscosity supersolution of
\[-w_{yy}\geq 0\hspace{0.5cm}\text{on }\mathbb{R}\times(0,\infty).\]
This is because for any smooth function $\phi$ touching $w$ from below at $(x_0,y_0)\in\mathbb{R}\times(0,\infty)$, we have
\[\phi_y(x_0,y_0)\geq\sup\limits_{\alpha\in\mathbb{R}}\left[\frac{1}{2}\phi_{xx}(x_0,y_0)+\alpha\phi_{xy}(x_0,y_0)+\frac{1}{2}\alpha^2\phi_{yy}(x_0,y_0)\right]\]
by the viscosity solution property. However, this implies $-\phi_{yy}(x_0,y_0)\geq 0$ because, otherwise, the right-hand-side would be arbitrarily large for large $\alpha$.

We conclude that for each fixed $x_0\in\mathbb{R}$, the function $y\mapsto w(x_0,y)$ is a viscosity supersolution
\[-w_{yy}\geq 0\hspace{0.5cm}\text{on }(0,\infty).\]
This, in turn implies that $y\mapsto w(x_0,y)$ is concave on $(0,\infty)$ for each fixed $x_0\in\mathbb{R}$. See, for example, Propositions 6.9 and Lemma 6.23 in \cite{Touzi2013}.

Finally, by the growth bounds in Lemma \ref{Lem:AuxiliaryValueBounds}, it is clear that $w$ is continuous at $(x_0,0)$, so we can conclude that for each $x_0\in\mathbb{R}$, the map $y\mapsto w(x_0,y)$ is concave on $[0,\infty)$.
\end{proof}

We can immediately apply this result to find conditions under which the the sequential time-consistent optimization problem (\ref{Eqn:SeqOptimization}) is concave.

\begin{cor}\label{Cor:Concavity}Suppose that Assumption~\ref{Assumption:Growth} holds. Let $g:\mathbb{R}^+\to\mathbb{R}$ be concave. Then the map
\[y\mapsto w(x,y)+g(y)\]
is concave for each $x\in\mathbb{R}$.\end{cor}
\begin{proof}This follows immediately from Proposition \ref{Prop:ConcaveValueFunction} and the concavity of $g$.\end{proof}

In order to construct an optimal stopping time as outlined in Section \ref{Subsection:Construction}, we first need to find a maximizing element $y\in[0,\infty)$ of the sequential time-inconsistent optimization problem (\ref{Eqn:SeqOptimization}). We next show a coercivity condition on $g$ which will provide existence of a minimizing element.

\begin{thm}\label{Thm:MaximalElement}
Suppose that Assumption~\ref{Assumption:Growth} holds, that $g$ is concave, and that there exists $\lambda <\beta-\gamma$ such that
\[g(y)\leq\lambda y.\]
Then for each $x\in\mathbb{R}$, the map
\[y\mapsto w(x,y)+g(y)\]
is concave and has a maximizing element $y\in[0,\infty)$.\end{thm}
\begin{proof}
Fix $x\in\mathbb{R}$. From Corollary \ref{Cor:Concavity}, we know the map $y\mapsto w(x,y)+g(y)$ is concave. We claim that this map is also bounded from above by a strictly decreasing linear function. Therefore, there exists a maximizing element $y^\star\in[0,\infty)$.

Let $y\in[0,\infty)$ and $\tau\in\mathcal{T}_y$ be arbitrary. Similarly to Lemma~\ref{Lem:AuxiliaryValueBounds}, we compute
\begin{eqnarray}
\mathbb{E}^x\left[f(W(\tau))\right] & \leq & \gamma + \gamma\mathbb{E}^x\left[|W(\tau)|\right]-\beta\mathbb{E}^x\left[W(\tau)^2\right]\nonumber\\
& \leq & \gamma(1+|x|)-\beta x^2+(\gamma-\beta)\mathbb{E}^x\left[\tau\right]\nonumber\\
& = & \gamma(1+|x|)-\beta x^2+(\gamma-\beta)y.\nonumber
\end{eqnarray}
Because $\tau$ was arbitrary, we conclude
\[w(x,y)\leq\gamma(1+|x|)-\beta x^2+(\gamma-\beta)y.\]
Together with the growth assumption of $g$
\[w(x,y)+g(y)\leq\gamma(1+|x|)-\beta x^2+(\lambda+\gamma-\beta)y.\]
Because $x$ is fixed and $\lambda<\beta-\gamma$, the map $y\mapsto w(x,y)+g(y)$ is bounded from above by a strictly decreasing linear function in $y$. Then the result follows.
\end{proof}

\section{Generalizations and Concluding Remarks}\label{Section:Generalization}

The goal of this section is to highlight various generalizations and connections. In particular, we extend the results of Section \ref{Section:EquivalentProblem} to more general diffusion processes driven by Brownian motion and also other types of non-linearities in the time-inconsistent stopping problem. We also highlight a connection between our PDE and the Monge-Ampere equation. This section is very informal and meant as a launching point for future research. As such, we do not reproduce analogues of the results in Section \ref{Section:Properties}, though we expect that with appropriate conditions on $f$ and $g$ analogous results to all major properties will hold.

\subsection{Extension to Diffusion Processes}\label{Subsection:Extension}

Consider a probability space $(\Omega,\mathcal{F},\mathbb{P})$, which admits a Brownian motion $W$ valued in $\mathbb{R}^d$. Let $\mathbb{F}$ be the $\mathbb{P}$-augmentation of the canonical filtration of $W$. Let $\mu:\mathbb{R}^n\to\mathbb{R}^n$ and $\sigma:\mathbb{R}^n\to\mathcal{M}(n,d)$ be Lipschitz continuous with linear growth bounds. Then for each $x\in\mathbb{R}^n$, there is a unique strong solution to the SDE
\begin{equation}\nonumber
X(t) = x + \int_0^t \mu(X(s))ds + \int_0^t \sigma(X(s))dW(s).
\end{equation}
Let $\mathbb{P}^x$ denote the probability measure on $X$ corresponding to starting at $x\in\mathbb{R}$.

Denote the infinitesimal generator of the diffusion process $X$ by
\begin{equation}\nonumber
\mathcal{L}_x\phi := \mu\cdot D\phi + \frac{1}{2}\text{Tr}\left[\sigma\sigma^\top D^2\phi\right].
\end{equation}

The main result of this section is as follows.
\begin{prop}
Consider the time-inconsistent optimal stopping problem given by
\begin{equation}\nonumber
v(x) := \sup\limits_{\tau\in\mathcal{T}}\mathbb{E}^x\left[f(X(\tau))\right]+ g(\mathbb{E}^x\left[\tau\right]).
\end{equation}
Furthermore, consider an auxillary time-consistent control problem defined by
\begin{equation}\label{Eqn:GeneralAuxFunction}
w(x,y):=\left\{\begin{array}{rl}
\sup\limits_{\alpha\in\mathcal{A}}&\mathbb{E}^{x,y}\left[f(X(\tau^\alpha))\right]\\
& dX = \mu(X)dt + \sigma(X)dW\\
& dY^\alpha = -dt + \alpha dW,\,Y^\alpha(0)=y\\
& \tau^\alpha = \inf\left\{t\geq 0\mid Y^\alpha(t)=0\right\}.
\end{array}\right.
\end{equation}
Then we have the identity
\begin{equation}\nonumber
v(x) = \sup\limits_{y\geq 0}\left[w(x,y)+g(y)\right].
\end{equation}
\end{prop}
\begin{proof}The proof is as before, except now a higher-dimensional version of martingale representation theorem must be used, see \cite{Touzi2013}. In particular, this makes the control be valued in $\mathbb{R}^d$ rather than one-dimensional.
\end{proof}

Then if $w$ is locally bounded, then as before we can identify it as a viscosity solution of an HJB.
\begin{cor}
Assume the axillary value function $w$ defined in \eqref{Eqn:GeneralAuxFunction} is locally bounded. Then it is a viscosity solution to the HJB
\begin{equation}\label{Eqn:GeneralAuxHJB}\nonumber
\left\{\begin{array}{rl}
u_y - \sup\limits_{\alpha\in\mathbb{R}^d}\left[\mathcal{L}_x u+\alpha^\top\sigma D_x u_y+\frac{1}{2}|\alpha|^2 u_{yy}\right] = 0 & \text{in }\mathbb{R}^n\times(0,\infty)\\
u = f & \text{on }\mathbb{R}^n\times\{y=0\}.
\end{array}\right.\end{equation}
\end{cor}

\subsection{Connection to Monge-Ampere Equation}

Consider the one-dimensional problem from Section \ref{Section:Formulation}. We showed in Section \ref{Section:Properties} that the auxiliary value function $w$ is the unique viscosity solution of the HJB equation (\ref{Eqn:AuxiliaryHJB}) which is concave in $y$ for each $x$ and satisfies certain growth conditions. We claim that if $w$ is smooth, then it is also a solution to a Monge-Ampere-like equation.

\begin{prop}
Let $w$ be the auxiliary value function which is a viscosity solution of (\ref{Eqn:AuxiliaryHJB}). If $w$ is smooth, then $w$ a solution of
\begin{equation}\label{Eqn:MongeAmpere}
\left\{\begin{array}{rl}
2u_yu_{yy} - \det D^2 u = 0 & \text{in }\mathbb{R}\times(0,\infty)\\
u = f & \text{on }\mathbb{R}\times\{y=0\}.
\end{array}\right.\end{equation}
\end{prop}

\begin{remark}
The PDE \eqref{Eqn:MongeAmpere} is not elliptic, so the viscosity solution theory of \cite{CrandallIshiiLions1992} does not immediately apply. However, there are many classical results on viscosity solution interpretations of the standard Monge-Ampere equation (see \cite{IshiiLions1990}). However, we focus on classical solution theory at the moment.
\end{remark}

\begin{remark}
We consider \eqref{Eqn:MongeAmpere} as ``Monge-Ampere-like'' in the sense that it can be re-written as
\[-\det\left(\begin{array}{cc}
u_{xx}-2u_y & u_{xy} \\
u_{xy} & u_{yy}
\end{array}\right) = 0.\]
That is, it involves the determinant of second-order linear partial differential operators. Unlike in the classical case, the second-order partial derivative in $x$ is replaced by a heat equation analogue. We also note that this is a degenerate form of the two-dimensional generalized Monge-Ampere equations studied in \cite{Schulz1990}. 
\end{remark}

\begin{proof}
Because $w$ is a smooth viscosity solution of \eqref{Eqn:AuxiliaryHJB}, it is a classical solution. Fix $(x_0,y_0)\in\mathbb{R}\times(0,\infty)$. Then
\[w_y(x_0,y_0) = \sup\limits_{\alpha\in\mathbb{R}}\left[\frac{1}{2}w_{xx}(x_0,y_0)+\alpha w_{xy}(x_0,y_0)+\frac{1}{2}\alpha^2 w_{yy}(x_0,y_0)\right].\]
\begin{enumerate}
\item Suppose that $w_{yy}(x_0,y_0)<0$. Then the supremum is achieved at
\[\alpha^* := -\frac{w_{xy}(x_0,y_0)}{w_{yy}(x_0,y_0)}.\]
We can compute
\[w_y(x_0,y_0) = \frac{1}{2}w_{xx}(x_0,y_0)-\frac{w_{xy}(x_0,y_0)^2}{2w_{yy}(x_0,y_0)}=\frac{w_{xx}(x_0,y_0)w_{yy}(x_0,y_0)-w_{xy}(x_0,y_0)^2}{2w_{yy}(x_0,y_0)},\]
and consequently
\[2w_y(x_0,y_0)w_{yy}(x_0,y_0) = \det D^2w(x_0,y_0).\]

\item Suppose that $w_{yy}(x_0,y_0)=0$. Then for the PDE to hold, we must have $w_{xy}(x_0,y_0)=0$. Then we can compute
\[2w_y(x_0,y_0)w_{yy}(x_0,y_0)-\det D^2w(x_0,y_0) = (2w_y(x_0,y_0)-w_{xx}(x_0,y_0))w_{yy}(x_0,y_0)+w_{xy}(x_0,y_0)^2 = 0,\]
because $w_{xy}(x_0,y_0)=w_{yy}(x_0,y_0)=0$.

\item Finally, suppose that $w_{yy}(x_0,y_0)>0$. Then taking $\alpha$ large, we can make the supremum arbitrarily large and contradict the fact that $w$ is a smooth solution of the PDE at $(x_0,y_0)$. Then this case never occurs.
\end{enumerate}
\end{proof}

More generally, the PDE in Section \ref{Subsection:Extension} can be written in a Monge-Ampere-like form. There may be room for further analysis of these PDEs directly. Furthermore, this form may be amenable to numerics, as there is a whole literature on wide-stencil monotone schemes for numerical solutions of the Monge-Ampere equation (see \cite{Oberman2008}).

\subsection{Extension to Other Non-linearities}\label{Subsection:GeneralInconsistentcies}

In Section \ref{Section:EquivalentProblem}, we alluded to the claim that the approach used in this paper can be extended to other non-linearities leading to time-inconsistent stochastic optimization problems. The same type of argument holds in more general cases, but the resulting control problem, analogous to (\ref{Defn:AuxProblem}), has a terminal constraint which is non-standard from the perspective of viscosity solutions.

To illustrate this point, consider the one-dimensional setting of Section \ref{Section:Formulation}, but a more general time-inconsistent problem depending non-linearly on the higher moments of $(\tau,W(\tau))$.

\begin{prop}
Let $h:\mathbb{R}^+\times\mathbb{R}\to\mathbb{R}^m$ be a smooth, bounded function. Consider the time-inconsistent optimal stopping problem given by
\begin{equation}\label{Eqn:GeneralTimeInconsistent}
v(x) := \sup\limits_{\tau\in\mathcal{T}}\mathbb{E}^x\left[f(W(\tau))\right]+ g(\mathbb{E}^x\left[h(\tau,W(\tau))\right]).
\end{equation}
Let $\mathcal{A}$ and $\mathcal{B}$ be the collections of one- and $m$-dimensional square-integrable, progressively-measurable processes respectively. Consider an auxiliary time-consistent control problem defined by
\begin{equation}\label{Eqn:GeneralAuxFunction2}
w(x,y,z):=\left\{\begin{array}{rl}
\sup\limits_{(\alpha,\beta)\in\mathcal{A}\times\mathcal{B}}&\mathbb{E}^{x,y,z}\left[f(W(\tau^\alpha))\right]\\
& dY^\alpha = -dt + \alpha dW,\,Y^\alpha(0)=y\\
& dZ^\beta = \beta dW,\,Z^\beta(0)=z\\
& \tau^\alpha = \inf\left\{t\geq 0\mid Y^\alpha(t)=0\right\}\\
& Z^\beta(\tau^\alpha) = h(\tau^\alpha,W(\tau^\alpha)).
\end{array}\right.
\end{equation}
Then we have the identity
\begin{equation}\nonumber
v(x) = \sup\limits_{(y,z)\in[0,\infty)\times\mathbb{R}^m}\left[w(x,y,z)+g(z)\right].
\end{equation}
\end{prop}
\begin{proof}
We provide a sketch of a proof following the same three-step procedure from Section \ref{Section:EquivalentProblem}.

\begin{enumerate}
\item For any $(x,y,z)\in\mathbb{R}\times[0,\infty)\times\mathbb{R}^m$, consider the collection of stopping time / control pairs defined as
\[\mathcal{T}_{x,y,z} := \left\{\tau\in\mathcal{T}\mid \mathbb{E}^x\left[\tau\right] = y,\,\mathbb{E}^x\left[h(\tau,W(\tau))\right]=z\right\}.\]
Then we define an auxiliary constrained stopping problem as
\[\tilde{w}(x,y,z) := \sup\limits_{\tau\in\mathcal{T}_{x,y,z}} \mathbb{E}^x\left[f(W(\tau))\right]\hspace{0.5cm}\text{for all }(x,y,z)\in\mathbb{R}\times [0,\infty)\times\mathbb{R}^m.\]
Note, we take $\tilde{w}(x,y,z)=-\infty$ if $\mathcal{T}_{x,y,z}=\emptyset$.
Then by the same argument as in Theorem \ref{Thm:AuxEquivalence}, we have
\[v(x) = \sup\limits_{(y,z)\in[0,\infty)\times\mathbb{R}^m}\left[\tilde{w}(x,y,z)+g(z)\right].\]

\item We next claim that $w=\tilde{w}$. This follows by applying the martingale representation theorem on $\tau$ to obtain $\alpha\in\mathcal{A}$ and on the vector $h(\tau,W(\tau))$ to obtain $\beta\in\mathcal{B}$.

\item Finally, given a maximizing $(y^*,z^*)\in[0,\infty)\times\mathbb{R}^m$, we can construct an $\epsilon$-suboptimal control to the generalized time-inconsistent problem (\ref{Eqn:GeneralTimeInconsistent}) starting from $x\in\mathbb{R}$ by constructing a suboptimal control of (\ref{Eqn:GeneralAuxFunction2}) starting from $(x,y^*,z^*)$.
\end{enumerate}

\end{proof}

For many $(x,y,z)\in\mathbb{R}\times[0,\infty)\times\mathbb{R}^m$, there may be no feasible controls. At these points, the auxiliary value function is formally taken to have the value $-\infty$. It is fairly straightforward to formally show that $w$ is a viscosity solution to a certain HJB at any point where it is finite, but dealing with the equality constraints is non-standard. We leave a PDE-based analysis of this problem to further work.

\section*{Acknowledgement}

The author would like to thank Erhan Bayraktar, Lawrence Craig Evans, Chandrasekhar Karnam, Goran Peskir, and Insoon Yang for their invaluable input on current approaches to time-inconsistent optimization and feedback on early drafts of this paper. He would also like to thank his anonymous reviewers for their detailed and constructive feedback.

\appendix
\section{Proof of Lemma \ref{Lem:ComparisonPrinciple}}

The proof of the comparison principle in Lemma \ref{Lem:ComparisonPrinciple} follows the typical doubling of variables argument. However, because the diffusion coefficient is unbounded from above, we must be careful about the application of the Crandall-Ishii lemma, as in \cite{DaLioLey2010}.

\begin{proof}[Proof of Lemma \ref{Lem:ComparisonPrinciple}]
Assume to the contrary that
\[\sigma := \sup\limits_{\mathbb{R}\times[0,\infty)}\left(u-v\right) > 0.\]
\begin{enumerate}
\item For any fixed $\epsilon,\lambda >0$, define
\[\Phi(x,y,t,s) := u(x,t)-v(y,s)-\frac{1}{2\epsilon^2}\left((x-y)^2+(t-s)^2\right)-\frac{\epsilon}{2}(x^2+y^2)-\lambda(t+s).\]
Because both $u$ and $v$ have the same asymptotic quadratic growth in $x$, their difference has linear growth in both $x$ and $y$. If we take $\lambda>0$ large enough, then there exists a point $(x_0,y_0,t_0,s_0)\in\mathbb{R}^2\times[0,T]^2$ such that
\[\Phi(x_0,y_0,t_0,s_0) = \max\limits_{\mathbb{R}^2\times[0,\infty)^2}\Phi(x,y,t,s).\]
We can find a sequence $\epsilon_n\to 0$ such that the corresponding maximizers $(x_n,y_n,t_n,s_n)$ satisfy
\[(x_n,y_n,t_n,s_n)\to(\hat{x},\hat{y},\hat{t},\hat{s})\in\mathbb{R}^2\times[0,T]^2.\]
For sufficiently small $\epsilon_n$, we have that
\[\Phi(x_n,y_n,t_n,s_n)\geq\sup\limits_{\mathbb{R}\times[0,\infty)}\Phi(x,x,t,t)\geq\frac{\sigma}{2}.\]
But then
\begin{eqnarray}
\limsup\limits_{n\to\infty}\left[\frac{1}{2\epsilon_n^2}\left((x_n-y_n)^2+(t_n-s_n)^2\right)\right]&\leq &\limsup\limits_{n\to\infty}\left[u(x_n,t_n)-v(y_n,s_n)-\frac{\epsilon_n}{2}(x_n^2+y_n^2)\right.\nonumber\\
& & \left.\hspace{1cm}-\lambda(t_n+s_n)-\frac{\sigma}{2}\right]\nonumber\\
& = & u^*(\hat{x},\hat{t})-v_*(\hat{y},\hat{s})-\lambda(\hat{t}+\hat{s})-\frac{\sigma}{2}.\nonumber
\end{eqnarray}
So we see $\hat{x}=\hat{y}$ and $\hat{t}=\hat{s}$. Finally, we have
\[u^*(\hat{x},\hat{t})-v_*(\hat{x},\hat{t})\geq\frac{\sigma}{2}-2\lambda\hat{t},\]
which contradicts $u\leq v$ on $t=0$ if $\hat{t}=0$. Therefore, for sufficiently small $\epsilon_n$, we have
\[\Phi(x,y,t,s)\text{ attains a maximum at }(x_n,y_n,t_n,s_n)\in\mathbb{R}^2\times(0,\infty)^2.\]

\item By the Crandall-Ishii lemma (see \cite{CrandallIshiiLions1992,Touzi2013}), for any $\rho >0$, there exists $A,B\in\mathcal{S}_2$ such that
\[\epsilon_n^{-2}(t_n-s_n)+\lambda\leq\sup\limits_{\alpha\in\mathbb{R}}\left[\frac{1}{2}a_{11}+\alpha a_{12}+\frac{1}{2}\alpha^2 a_{22}\right],\]
\[\epsilon_n^{-2}(t_n-s_n)-\lambda\geq\sup\limits_{\alpha\in\mathbb{R}}\left[\frac{1}{2}b_{11}+\alpha b_{12}+\frac{1}{2}\alpha^2 b_{22}\right],\]
and
\[-\left(\rho^{-1}+|M|\right)I\leq\left[\begin{array}{cc}A&0\\0&-B\end{array}\right]\leq M+\rho M^2,\]
where
\[M := \epsilon_n^{-2}\left[\begin{array}{cc}
I & -I \\
-I & I
\end{array}\right]
+\epsilon_n\left[\begin{array}{cccc}
1 & 0 & 0 & 0 \\
0 & 0 & 0 & 0 \\
0 & 0 & 1 & 0 \\
0 & 0 & 0 & 0 
\end{array}\right].\]
Subtracting the first two inequalities above, we obtain
\[2\lambda\leq \sup\limits_{\alpha\in\mathbb{R}}\left[\frac{1}{2}a_{11}+\alpha a_{12}+\frac{1}{2}\alpha^2 a_{22}\right]-\sup\limits_{\beta\in\mathbb{R}}\left[\frac{1}{2}b_{11}+\beta b_{12}+\frac{1}{2}\beta^2 b_{22}\right].\]
That is, for any $\eta >0$, there exists $\alpha\in\mathbb{R}$ such that for all $\beta\in\mathbb{R}$, we have
\[2\lambda-\eta\leq \frac{1}{2}a_{11}+\alpha a_{12}+\frac{1}{2}\alpha^2 a_{22}-\frac{1}{2}b_{11}-\beta b_{12}-\frac{1}{2}\beta^2 b_{22}.\]
Applying the matrix inequality above with $\rho:=\epsilon_n^2$ to the vector $(1,\alpha,1,\beta)$, we obtain
\[a_{11}+2\alpha a_{12}+\alpha^2 a_{22}-b_{11}-2\beta b_{12}-\beta^2 b_{22} \leq
3\epsilon_n^{-2}(\alpha-\beta)^2 + 6\epsilon_n+2\epsilon_n^4.\]
Taking $\eta := \lambda$ and $\beta :=\alpha$, we obtain
\[\lambda \leq 6\epsilon_n+2\epsilon_n^4,\]
which is a contradiction for $\epsilon_n$ sufficiently small.
\end{enumerate}
\end{proof}

\bibliographystyle{siam}

\bibliography{TimeInconsistentStopping}

\end{document}